\documentclass[12pt, reqno]{amsart}
\usepackage[dvipsnames]{xcolor}
\usepackage{mathtools}
\mathtoolsset{showonlyrefs}
\usepackage{etoolbox}
\usepackage{blindtext}
\usepackage{latexsym}
\usepackage[pagewise]{lineno}
\usepackage[utf8]{inputenc}
\usepackage{exscale}
\usepackage{amsmath}
\usepackage{amssymb}
\usepackage{amsfonts}
\usepackage{mathrsfs}
\usepackage{amsbsy}
\usepackage{relsize}
    
    \mathchardef\Re="023C
\mathchardef\Im="023D

\usepackage{comment}
\PassOptionsToPackage{reqno}{amsmath}

\definecolor{mypink1}{rgb}{0.858, 0.188, 0.478}
\definecolor{mypink2}{RGB}{219, 48, 122}
\definecolor{mypink3}{cmyk}{0, 0.7808, 0.4429, 0.1412}
\definecolor{mygray}{gray}{0.6}
\definecolor{venetianred}{rgb}{0.78, 0.03, 0.08}
\definecolor{sapphire}{rgb}{0.03, 0.15, 0.4}
\definecolor{utahcrimson}{rgb}{0.83, 0.0, 0.25}
\definecolor{trueblue}{rgb}{0.0, 0.45, 0.81}
\definecolor{carminered}{rgb}{1.0, 0.0, 0.22}
\definecolor{cobalt}{rgb}{0.0, 0.28, 0.67}
\definecolor{cornflowerblue}{rgb}{0.39, 0.58, 0.93}
\definecolor{darkmagenta}{rgb}{0.55, 0.0, 0.55}
\definecolor{electricultramarine}{rgb}{0.25, 0.0, 1.0}
\definecolor{falured}{rgb}{0.5, 0.09, 0.09}
\definecolor{hanpurple}{rgb}{0.32, 0.09, 0.98}
\definecolor{mahogany}{rgb}{0.75, 0.25, 0.0}
\definecolor{oucrimsonred}{rgb}{0.6, 0.0, 0.0}
\definecolor{persianblue}{rgb}{0.11, 0.22, 0.73}
\definecolor{rufous}{rgb}{0.66, 0.11, 0.03}
\definecolor{uablue}{rgb}{0.0, 0.2, 0.67}
\definecolor{zaffre}{rgb}{0.0, 0.08, 0.66}
\definecolor{carmine}{rgb}{0.59, 0.0, 0.09}

\usepackage[colorlinks, linkcolor=carminered, citecolor=electricultramarine, urlcolor=mahogany, hypertexnames=false]{hyperref}

\usepackage{esint} 
\usepackage{amssymb} 
\usepackage{stmaryrd}
\usepackage{cite} 
\usepackage[top=2.5cm, bottom=2.5cm, left=2.5cm, right=2.5cm]{geometry}
\def\N{\mathbb{N}}
\def\R{\mathbb{R}}

\def\C{\mathbb{C}}

\numberwithin{equation}{section}

\newtheorem{thm}{Theorem}[section]

\newtheorem{lem}[thm]{Lemma}
\newtheorem{prop}[thm]{Proposition}

\theoremstyle{remark}
\newtheorem{remark}[thm]{Remark}
\newtheorem*{exam*}{Examples}

\numberwithin{equation}{section}

\begin{document}

\title[Non-radial BLOW-UP]{Non-radial Blow-up for a mass-critical fourth-order inhomogeneous  nonlinear Schr\"odinger equation}

\author{Ruobing Bai}
\address{School of Mathematics and Statistics\\
Henan University\\
Kaifeng 475004, China.}
\email{\it{\textcolor{blue}{baimath@hotmail.com}}}
\author{Mohamed Majdoub}
\address{Department of Mathematics, College of Science, Imam Abdulrahman Bin Faisal University, P. O. Box 1982, Dammam, Saudi Arabia.\newline Basic and Applied Scientific Research Center, Imam Abdulrahman Bin Faisal University, P.O. Box 1982, 31441, Dammam, Saudi Arabia.}
\email{\sl {\textcolor{blue}{mmajdoub@iau.edu.sa}}}
\email{\sl {\textcolor{blue}{med.majdoub@gmail.com}}}
\author{Tarek Saanouni}
\address{Department of Mathematics, College of Science, Qassim University, Buraydah, Kingdom of Saudi Arabia.}
\email{\sl{\textcolor{blue}{tarek.saanouni@ipeiem.rnu.tn}}}
\email{\sl{\textcolor{blue}{ t.saanouni@qu.edu.sa}}}
\thanks{}

\subjclass[2020]{Primary: 35Q55; Secondary: 35B44}


\keywords{Inhomogeneous fourth-order nonlinear Schr\"odinger equation, $L^2-$critical, blow-up, localized virial identity, non-radial solutions}

\begin{abstract}\noindent
We investigate the blow-up for a fourth-order Schr\"odinger equation with a mas-critical focusing inhomogeneous nonlinearity. We prove the finite/infinite-time blow-up of non-radial solutions with negative energy. Our result serves as a valuable complement to the existing literature and offers an improvement in our understanding of the subject matter. 
\end{abstract}

\maketitle

\section{Introduction}
\label{S1}

In this work, we are concerned with the initial value problem for the $L^2-$critical inhomogeneous fourth-order Schr\"odinger equation with focusing power-type nonlinearity 
\begin{equation}
\left\{
\begin{array}{ll}
{\rm i}\partial_t u-\Delta^2 u+ \nu \Delta u=-|x|^{-b}|u|^{\frac{8-2b}{N}}u,\\
u(0,x)=u_0(x)\in H^2(\R^N),
\label{BINLS-mu}\tag{IBNLS}
\end{array}
\right.
\end{equation}
where $b>0$ and $\nu\geq 0$.

The fourth-order nonlinear equations of Schr\"odinger type were introduced by Karpman \cite{Karp1996, Karp-1996} and Karpman–Shagalov \cite{KS2000} in the physics literature. In particular, high-order Schr\"odinger type equations can be seen as an interpolation  between Schr\"odinger and relativistic equations. See also \cite{FIP} and the references therein for further interpretations. 

It is worth noting that over the past two decades, there has been a substantial researches dedicated to exploring fourth-order nonlinear Schr\"odinger equations. A comprehensive overview of the advancements in this field can be found in \cite{BKS, Paus1, Paus2, Paus} and the references cited therein.

As for the classical NLS equation, the fourth-order NLS equation \eqref{BINLS-mu} enjoys the conservation of the following quantities called respectively mass and energy 
\begin{eqnarray*}
M[u(t)]&:=&\int_{\R^N}\,|u(t,x)|^2\,dx,\\
E[{u(t)}]&:=&\int_{\R^N}\,\Big(|\Delta u(t,x)|^2+\nu|\nabla u(t,x)|^2\Big)\,dx-\frac{2N}{2N+8-b}\int_{\R^N} |x|^{-b}|u(t,x)|^{1+\frac {8-2b}{N}}dx.
\end{eqnarray*}

In \cite{GP2020,lz}, the authors consider the following inhomogeneous biharmonic Schr\"odinger equation 

\begin{equation}
    \label{IBNLS}
{\rm i}\partial_t u-\Delta^2 u=\pm|x|^{-b}|u|^{p-1}u.    
\end{equation}

They showed the local well-posedness for \eqref{IBNLS} in the energy space. 

\begin{thm}(\cite[Corollary 1.6]{lz})
\label{LWP}
Let $N\geq 1$, $0<b<\min\left\{\frac{N}{2}, 4\right\}$, and \newline$0<p-1<\frac{2(4-b)}{(N-4)_+}$\,\footnote{We use the notation $\kappa_+ :=\max(\kappa,0)$ with the convention $0^{-1}=\infty$.}. If $u_0\in H^2(\R^N)$, then there exists a unique local solution $u\in C([-T,T]; H^2)$ of \eqref{IBNLS} with $u(0)=u_0$, where $T:=T(\|u_0\|_{H^2}, b,N,p)>0$.
\end{thm}

Regarding blow-up phenomena, the case of $b=0$ has been extensively examined in \cite{BL2017} for radially symmetric initial data with negative energy. The study reveals that if $\nu>0$, a finite-time blow-up is guaranteed to transpire. However, when $\nu=0$, the situation becomes more intricate as either a finite-time blow-up or a blow-up at infinity becomes possible outcomes. The case with inhomogeneity, characterized by $b>0$, was addressed in \cite{Dinh} specifically for the scenario where $b=2$ and $\nu>0$. The study demonstrates the occurrence of a finite-time blow-up even in the absence of spherical symmetry assumptions on the initial data, provided that the initial energy is negative. To the best of our knowledge, the case where $\nu=0$ and $b\geq 0$ presents an intriguing question that has yet to be fully addressed. Despite extensive research in the field, a conclusive understanding of the dynamics and behavior of solutions under these conditions remains elusive. \\
In line with the spirit of \cite[Theorem 3, p. 507]{BL2017}, we have successfully derived the following blow-up result.
\begin{thm}
\label{BLOW1}
Let $N\geq 1$, $0<b<\min\left\{\frac{N}{2}, 4\right\}$, $\nu\geq 0$, and $u_0\in H^2(\R^N)$ with $E[u_0]<0$. Let $u\in C([0, T^*); H^2)$ be the maximal solution of \eqref{BINLS-mu}. Then we have the following.
\begin{itemize}
    \item[(i)] If $\nu>0$, then $u(t)$ blows-up in finite time.
    \item[(ii)] If $\nu=0$, then $u(t)$ either blows up in finite or in infinite time in the sense that 
    \begin{equation}
        \label{Blow-infinity}
        \|\Delta u(t)\|\gtrsim  t^{2},\quad\mbox{for}\quad t\gg 1.
    \end{equation}
\end{itemize}
\end{thm}

In view of the above result, some comments are in order.

~\begin{itemize}
\item[$\blacktriangleright$]
The main contribution here is to remove the radial assumption in \cite{BL2017}.
\item[$\blacktriangleright$] 
The first part of the above result extends \cite{Dinh} to the case $b\neq2$.
\item[$\blacktriangleright$] 
To the best authors knowledge, the finite-time blow-up for the mass-critical homogeneous biharmonic NLS, namely \eqref{BINLS-mu} for $b=0$ and $\nu\leq0$, is still open even for radial data.
\item[$\blacktriangleright$]
The biharmonic NLS has no variance identity as in the classical NLS. This makes the negative sign of the Morawetz potential term not sufficient to ensure the finite-time blow-up. This is an essential difference between NLS and BNLS.
\item[$\blacktriangleright$]
The proof relies on the assumption that $b$ is greater than zero, and therefore, the aforementioned result cannot be applied to the special scenario when $b$ equals zero.
\item[$\blacktriangleright$]
During the proof, a carefully chosen cut-off function is employed within the localized virial identity, which enables us to effectively manage the remaining terms by utilizing the decay property of the inhomogeneous term $|x|^{-b}$. This approach not only aids in handling the remainder but also assigns a specific sign to the derivative of the virial identity. Ultimately, we arrive at an ordinary differential inequality that possesses solutions which are not globally defined.
\item[$\blacktriangleright$]
In equation \eqref{zr'2}, the case where $\nu=0$ holds special significance due to its delicate nature. When $\nu$ takes on this specific value, the dispersion term $-\nu\|\nabla u\|_{L^2}^2$ disappears entirely. This implies that the contribution of the dispersion effect, which is captured by the gradient of the solution $u$, is absent in the equation. Consequently, the dynamics  is altered, potentially leading to different behavior or properties of the solutions compared to other cases where $\nu$ is nonzero.
\item[$\blacktriangleright$]
The intercritical case has been recently addressed in \cite{rbts}.
\end{itemize}

The outline of the article is as follows. In Section \ref{S2}, we start by introducing some notations, recalling some standard identities followed by several useful estimates and preliminary results. Finally, in Section \ref{S3}, we give the proof of our main result Theorem \ref{BLOW1}. 
\section{Preliminaries}
\label{S2}
\subsection{Notations}
\begin{itemize}
    \item[$\triangleright$] We write $X\lesssim Y$ or $Y\gtrsim X$ to denote the estimate $X\leq CY$ for some constant $C>0$. \item[$\triangleright$] The notation $X\backsim Y$ means that $X\lesssim Y$ and $Y\lesssim X$. 
\item[$\triangleright$] When $\Omega=\R^N$, we abbreviate $L^2(\Omega)$ as $L^2$ and $\|\cdot \|_{L^2(\R^N)}$ as $\|\cdot\|$. 
\item[$\triangleright$] For any $\beta>0$, we use the notation $O(R^{-\beta})$ to denote an infinitely small quantity of the same order as $R^{-\beta}$, that is $\frac{O(R^{-\beta})}{R^{-\beta}}\rightarrow \ell\neq 0$, as $R\rightarrow +\infty$.
\item[$\triangleright$] We define the Fourier transform of a function $f$ by:
\begin{align*}
\widehat{f}(\xi)=\mathscr{F}f(\xi):=\int_{\R^N}e^{-ix\cdot \xi}f(x)dx.
\end{align*}
\item[$\triangleright$] For a complex number $z\in \C$, we use $\Im(z)$ to denote its imaginary part. 
\item[$\triangleright$] Throughout the whole paper, the letter $C$ will denote different positive constants which are not important in our analysis and may vary line by line.
\end{itemize}

\subsection{Useful results}

In this section, we collect some useful tools needed in the proof of our main result. Let us begin with a classical Gagliardo-Nirenberg interpolation inequality.
\begin{prop}
    \label{GNI-Classical}
    The following inequality holds true
    \begin{equation}
        \label{GNI-RN}
        \|\nabla u\|\lesssim \|\Delta u\|^{1/2}\,\|u\|^{1/2}.
    \end{equation}
\end{prop}
\begin{proof}
    For reader's convenience, we give a simple proof here. Using Parseval's identity for the Fourier transform, we write for $M>0$ to be chosen later
   \begin{eqnarray*}
    \|\nabla u\|^2&\lesssim& \int_{\R^N}\,|\xi|^2\,|\widehat{u}(\xi)|^2\,d\xi\\
    &\lesssim&  \int\limits_{|\xi|\leq M}\,|\xi|^2\,|\widehat{u}(\xi)|^2\,d\xi+\int\limits_{|\xi|>M}\,|\xi|^2\,|\widehat{u}(\xi)|^2\,d\xi\\    
    &\lesssim& M^2 \|u\|^2+M^{-2}\|\Delta u\|^2.
   \end{eqnarray*}  
We conclude the proof by choosing 
$M=\|\Delta u\|^{\frac{1}{2}}\;\|u\|^{-\frac{1}{2}}.$
\end{proof}
In general, inequality \eqref{GNI-RN} does not hold for bounded domains $\Omega$ unless $u$ is in $W^{2,2}_0(\Omega)$. The question of whether \eqref{GNI-RN} holds for exterior domains $\Omega$ without assuming that $u$ has zero trace on $\partial\Omega$ was addressed in \cite{CM}.

\begin{prop}(\cite[Theorem 2.1]{CM})
    \label{GNI-Ext}
    Let $\Omega\subset\R^N$ be an exterior domain having the cone property. Then\footnote{For $k\in\N$, the notation $\|\mathbf{D}^k w\|_{L^2(\Omega)}$ stands for $\displaystyle\sup_{|\alpha|=k}\,\|\partial^{\alpha} w\|_{L^2(\Omega)}$.}
    \begin{equation}
        \label{GNI}
\|\mathbf{D}w\|_{L^2(\Omega)}\leq\, C(\Omega)\, \|\mathbf{D}^2w\|_{L^2(\Omega)}^{1/2}\,\|w\|_{L^2(\Omega)}^{1/2}.
    \end{equation}
\end{prop}
\begin{remark}
~{\rm \begin{itemize}
    \item[(i)]  Owing to the cone property of the exterior of an Euclidean ball in $\R^N$, we can deduce from \eqref{GNI} that
\begin{equation}
        \label{GNI-R}
        \|\nabla w\|_{L^2(\Omega_R)}\leq C\,\|\Delta w\|_{L^2(\Omega_R)}^{1/2}\,\|w\|_{L^2(\Omega_R)}^{1/2},
    \end{equation}
    where $R>0$ and $\Omega_R=\left\{\,x\in\R^N;\;\;\; |x|>R\,\right\}.$ 
    
    \item[(ii)] A scaling argument reveals that the constant $C$ in \eqref{GNI-R} is independent of $R$.
    \item[(iii)] A more general statement and other valuable extensions can be found in \cite{CM} and the literature cited therein.
    \end{itemize}}
    \end{remark}
Next, we provide an interpolation estimate in the spirit of \cite[Lemma 2.1]{CF2022}.
\begin{lem}\label{1}
Let $0<b<4$, $N\geq 1$ and $\psi$ be a positive real-valued function.
\begin{enumerate}
\item[1.]
If $N\geq 5$ and $\psi^{\frac{2}{4-b}} \in W^{2,\infty}$, then for all $u\in H^2$ we have
\begin{align}
\int_{\R^N}\psi|u|^{\frac{8-2b}{N}+2}dx\lesssim \left\|\Delta\Big(\psi^{\frac{2}{4-b}}u\Big)\right\|^{\frac{4-b}{2}}\, \|u\|^{\frac{8-2b}{N}+\frac b2}.\label{n=5}
\end{align}
\item[2.] {If $N=4$  and $\psi^{\frac{2}{4-b/2}} \in W^{2,\infty}$, then for all $u\in H^2$ we have
\begin{align}
\int_{\R^N}\psi|u|^{\frac{8-2b}{N}+2}dx
&\lesssim\left\|\psi^\frac2{4-\frac b2}u\right\|_{H^2}^{2-\frac b4}\, \|u\|^{2-\frac b4}.\label{n=4}
\end{align}
}
\item[3.] {If $N=3$ and $\psi^{\frac{2}{4-b}} \in W^{2,\infty}$, then for all $u\in H^2$ we have
\begin{align}
\int_{\R^N}\psi|u|^{\frac{8-2b}{N}+2}dx
&\lesssim\|u\|_{H^2}^{\frac{4-b}{6}}\|\psi^\frac{2}{4-b}u\|_{H^2}^\frac{4-b}{2}\|u\|^2.\label{n=3}
\end{align}
}
\item[4.] 
If $N=2$ and $\psi \in L^{\infty}$, then for all $u\in H^2$ we have
\begin{align}
\int_{\R^N}\psi|u|^{6-b}dx
\lesssim \|\psi\|_{\infty}\|u\|^2\|\nabla u\|^{4-b}.\label{n=2}
\end{align}
\item[5.] 
If $N=1$ and $\psi \in L^{\infty}$, then for all $u\in H^2$ we have
\begin{align}
\int_{\R^N}\psi|u|^{10-2b}dx
&\lesssim\|\psi\|_{\infty}\|u\|^{6-b}\|\nabla u\|^{4-b}.\label{n=1}
\end{align}
\end{enumerate}

\end{lem}

\begin{proof}

By H\"older's inequality, we have
\begin{align*}
\int_{\R^N}\psi|u|^{\frac{8-2b}{N}+2}dx&=\int_{\R^N}\psi|u|^{2-\frac b2}|u|^{\frac{8-2b}{N}+\frac b2}dx\\
&\lesssim \|\psi|u|^{2-\frac b2}\|_{L^\frac{2}{2-(\frac{8-2b}{N}+\frac b2)}}\||u|^{\frac{8-2b}{N}+\frac b2}\|_{L^\frac{2}{\frac {8-2b}{N}+\frac b2}}\\
&=\|\psi^{\frac {2}{4-b}}u\|_{L^\frac{2N}{N-4}}^{\frac{4-b}{2}}\|u\|^{\frac{8-2b}{N}+\frac b2},
\end{align*}
provided that 
\begin{align}
0<b\leq4,\quad N\geq5,\quad 0<\frac {8-2b}{N}+\frac b2<2.\label{n5}
\end{align}
Clearly, \eqref{n5} is satisfied for $N\geq5$ and $0<b<4$. Thus, using Sobolev embedding, we get
\begin{align*}
\int_{\R^N}\psi|u|^{\frac{8-2b}{N}+2}dx
&\lesssim\|\Delta(\psi^{\frac {2}{4-b}}u)\|^{\frac{4-b}{2}}\|u\|^{\frac{8-2b}{N}+\frac b2}\\
&\lesssim \Big(\|\Delta(\psi^{\frac{2}{4-b}})u\|+\|\nabla(\psi^{\frac{2}{4-b}})\cdot\nabla u\|+\|\psi^{\frac{2}{4-b}}\Delta u\|\Big)^{\frac{4-b}{2}}\|u\|^{\frac{8-2b}{N}+\frac b2}.
\end{align*}
This finishes the proof of \eqref{n=5}.

Suppose now that {$N=4$. Again, by using  H\"older's inequality and Sobolev embedding, we infer
\begin{align*}
\int_{\R^N}\psi|u|^{\frac{8-2b}{N}+2}dx
&=\int_{\R^N}\Big(\psi^\frac1{2-\frac b4}|u|\Big)^{2-\frac b4}|u|^{2-\frac b4}dx\\
&\leq\left\|\Big(\psi^\frac1{2-\frac b4}|u|\Big)^{2-\frac b4}\right\|_{L^{\frac8b}}\|u\|^{2-\frac b4}\\
&\leq\left\|\psi^\frac1{2-\frac b4}u\right\|_{L^{\frac{2(8-b)}b}}^{2-\frac b4}\|u\|^{2-\frac b4}\\
&\lesssim\|\psi^\frac2{4-\frac b2}u\|_{H^2}^{2-\frac b4}\|u\|^{2-\frac b4}.
\end{align*}
Arguing as above, we handle the case $N=3$ as follows:
\begin{align*}
\int_{\R^N}\psi|u|^{\frac{8-2b}{N}+2}dx
&\leq\left\|\psi|u|^\frac{8-2b}{N}\right\|_{L^\infty}\|u\|^2\\
&\leq\|u\|_{L^\infty}^{\frac{4-b}{6}}\left\|\psi^\frac{2}{4-b}u\right\|_{L^\infty}^\frac{4-b}{2}\|u\|^2\\
&\lesssim\|u\|_{H^2}^{\frac{4-b}{6}}\, \left\|\psi^\frac{2}{4-b}u\right\|_{H^2}^\frac{4-b}{2}\|u\|^2.
\end{align*}
}
Now, if $N=2$, it is sufficient to use H\"older's inequality and Sobolev embedding as follows
\begin{align*}
\int_{\R^N}\psi|u|^{6-b}dx\lesssim \|\psi\|_{\infty}\|u\|_{6-b}^{6-b}
\lesssim \|\psi\|_{\infty}\|u\|^2\|\nabla u\|^{4-b}.
\end{align*}
Finally, if $N=1$, it is sufficient to use H\"older's inequality and Sobolev embedding as above.

The proof of Lemma \ref{1} is completed.
\end{proof}
\subsection{Localized virial Identity}
Let $\varphi : \R^N\to \R$ be a cut-off function. The virial quantity associated to a $H^2-$ solution of \eqref{BINLS-mu} is defined as
\begin{equation}
\label{Vir}
\mathcal{Z}_\varphi(t):= \Im\left(\int_{\R^N} \nabla\varphi\cdot\nabla u(t)\bar u(t)dx\right).
\end{equation}
We have the following virial identity.
\begin{lem}
\label{ZR-deriv}
Let $N\geq 3$ and $\nu\geq 0$. Suppose that $u\in C([0,T); H^2)$ is a solution of \eqref{BINLS-mu}. Then, for any $t\in [0,T)$, we have
\begin{align}
\frac{d}{dt}\mathcal{Z}_\varphi(t)
=&-4\sum_{j,k=1}^N\int_{\R^N}\partial_{jk}\Delta\varphi\partial_ju\partial_k\bar udx+\int_{\R^N}\Delta^3\varphi|u|^2dx+8\sum_{i,j,k=1}^N\int_{\R^N}\partial_{jk}\varphi\partial_{ik}u\partial_{ij}\bar udx\nonumber\\
&-2\int_{\R^N}\Delta^2\varphi|\nabla u|^2dx-\nu\int_{\R^N}\Delta^2\varphi|u|^2dx+4\nu\sum_{j,k=1}^N\int_{\R^N}\partial_{jk}\varphi\partial_ju\partial_k\bar udx\label{382}\\
&-\frac{8-2b}{N+4-b}\int_{\R^N}\Delta\varphi|x|^{-b}|u|^{\frac{8-2b}{N}+2}dx+\frac{2N}{N+4-b}\int_{\R^N}\nabla\varphi\cdot\nabla(|x|^{-b})|u|^{\frac{8-2N}{N}+2}dx.\nonumber
\end{align}

\end{lem}
\begin{remark}
{\rm The proof of Lemma \ref{ZR-deriv} is given in \cite[Lemma 2.2]{Dinh} for $b=0$, and in \cite[Lemma 3.3]{Saan} for $b>0$. The details are therefore omitted.}
\end{remark}
For the sake of clarity and completeness, we recall the following radial identities:
\begin{align}
\nabla_x&=\frac x{r}\partial_r,\label{rd2}\\
\Delta_x&=\partial_r^2+\frac{N-1}r\partial_r,\label{rd3}\\
\partial_j\partial_k&=\Big(\frac{\delta_{jk}}r-\frac{x_jx_k}{r^3}\Big)\partial_r+\frac{x_jx_k}{r^2}\partial_r^2.\label{rd1}
\end{align}

We introduce, for $R>0$, the radial cut-off function  
\begin{equation}
\label{phiR}
\phi_R(r)=R^2\phi\left(\frac rR\right),\; r=|x|,
\end{equation}
where 
$$
\phi(r)=\int_{0}^r\chi(s)ds,
$$
and $\chi : [0,\infty)\to [0,\infty)$ is a smooth function (at least $C^3$) satisfying
\begin{equation*}
 \chi(r)=\left\{
		\begin{array}{ccc}
		 \displaystyle 2r&\text{if} & 0\leq r\leq 1, \\
		2r-2(r-1)^k  &\text{if}& 1<r\leq 1+k^{\frac 1{1-k}},\\
  \chi'(r)<0&\text{if}&1+k^{\frac 1{1-k}}<r<2,\\
  0 &\text{if}& r\geq 2.
		\end{array}
		\right.
\end{equation*}
Here $k\geq 4$ is an integer that will be chosen later (sufficiently large). We summarize some properties of the function $\phi_R$ in the following lemma.
\begin{lem}
\label{phi-R}
We have
\begin{equation} 
  \partial_r\phi_R(r)=2r ,\quad  \partial_r^2\phi_R(r)=2\;\;\mbox{for}\;\; r\leq R,\label{20233211}
\end{equation}
\begin{equation}\label{20233213}
 \partial_r\phi_R(r)-r\partial_r^2\phi_R(r)\geq0,\quad \partial_r\phi_R(r)\leq2r,\quad \partial_r^2\phi_R(r)\leq2,
\end{equation}
\begin{equation}
\label{nabla-phi-R}
    \|\nabla^j \phi_R\|_{L^\infty}\lesssim R^{2-j},\; 0\leq j\leq 6,
\end{equation}
and
\begin{align}
\mbox{supp}\left(\nabla^j \phi_R\right)\subset 
\begin{aligned}\left\{
		\begin{array}{ccc}
		\Big\{\;|x|\leq 2R\;\Big\} &\text{if} & j=1,2, \\
		\Big\{\; R\leq |x|\leq 2R\; \Big\} &\text{if}& 3\leq j\leq 6.
		\end{array}
		\right.  
  \end{aligned}
\end{align}
\end{lem}
The proof of Lemma \ref{phi-R} can be found in \cite{Dinh} for instance.

For later purpose, we define the functions
\begin{align}
\Phi_{1, R}&=8\left(2-\frac{\partial_r\phi_R}{r}\right),\label{fi1r}\\
\Phi_{2, R}&=\frac{2}{N+4-b}\bigg[(4-b)(2-\partial_r^2\phi_R)+(4N-4+b)(2-\frac{\partial_r\phi_R}{r})\bigg].\label{fi2r}
\end{align}
By \eqref{20233211}, we have for $r\leq R$, $\Phi_{1, R}=\Phi_{2, R}=0$.
It is clear that for any $r\geq 0$, 
\begin{align}\label{491}
|\Phi_{2, R}(r)|\lesssim 1.
\end{align}
The following estimates for $\Phi_{2, R}$ are crucial in the proof of our main results.
\begin{lem} 
\label{Claim}
Let $\alpha>0$. Then, we have

\begin{equation}
\label{nabla-Phi-R}
\|\nabla\left(\Phi_{2, R}^{\alpha} \right)\|_{L^\infty}\lesssim\frac 1{R},
\end{equation}
and
   \begin{align}
\|\Delta \left(\Phi_{2, R}^{\alpha} \right) \|_{L^\infty}\lesssim\frac 1{R^2}.\label{clm}
\end{align} 
\end{lem}
\begin{proof}
The estimates \eqref{nabla-Phi-R} and \eqref{clm} are straightforward for $r\leq R$ or $r\geq 2R$, since in both cases $\Phi_{2, R}$ is constant thanks to  \eqref{20233211}. In the sequel, we focus on the region $R<r<2R$.

$\bullet$ If $R<r\leq R(1+k^{\frac 1{1-k}})$, we have 
\begin{align*}
\partial_r\phi_R=2R\Big[\frac rR-(\frac rR-1)^{k}\Big], \quad\partial_r^2\phi_R=2-2k(\frac rR-1)^{k-1}.
\end{align*}
Thus, we obtain
\begin{align*}
\Phi_{2, R}^{\alpha}=(\frac rR-1)^{\alpha(k-1)}\left(\frac{4}{N+4-b}[k(4-b)+(4N-4+b)(1-\frac Rr)]\right)^{\alpha}.
\end{align*}

Since $1<\frac rR<1+k^{\frac 1{1-k}}<2$, then $\frac 12<\frac Rr<1$. Therefore
\begin{align}
\Big|\frac{N-1}r\partial_r\Phi_{2, R}^{\alpha}\Big|&=\Big|\frac{N-1}r\partial_r\Big[(\frac rR-1)^{\alpha(k-1)}\Big(\frac{4}{N+4-b}[k(4-b)+(4N-4+b)(1-\frac Rr)]\Big)^{\alpha}\Big]\Big|\nonumber\\
&\lesssim \frac 1{R^2}(\frac rR-1)^{\alpha(k-1)-1}+\frac{1}{r^2}(\frac rR-1)^{\alpha(k-1)}\nonumber\\
&\lesssim \frac 1{R^2},\label{-12.9}
\end{align}
where we assume  $k>1+\frac{1}{\alpha}$. This, in particular, gives \eqref{nabla-phi-R}. Similarly as above, we infer
\begin{align}
\Big|\partial_r^2\Phi_{2, R}^{\alpha}\Big|
=&\Big|\partial_r^2\Big[(\frac rR-1)^{\alpha(k-1)}\Big(\frac{4}{N+4-b}[k(4-b)+(4N-4+b)(1-\frac Rr)]\Big)^{\alpha }\Big]\Big|\nonumber\\
\lesssim& \frac 1{R^2}(\frac rR-1)^{\alpha(k-1)-2}+\frac 1{r^2}(\frac rR-1)^{\alpha(k-1)-1}+\frac 1{r^2}(\frac rR-1)^{\alpha(k-1)}\nonumber\\
\lesssim&\frac 1{R^2},\label{02.9}
\end{align}
provided that $k>2+\frac{1}{\alpha}$. Collecting the estimates \eqref{-12.9}-\eqref{02.9} and using \eqref{rd3}, we get
\begin{align*}
|\Delta \Phi_{2, R}^{\alpha} |\lesssim\frac 1{R^2}.
\end{align*}

$\bullet$ If $R(1+k^{\frac 1{1-k}})<r<2R$, we have
$$
\Phi_{2, R}=\frac{2}{N+4-b}\left[(4-b)\left(2-\chi'\left(\frac rR\right)\right)+(4N-4+b)\left(2-\frac R r\chi\left(\frac rR\right)\right)\right].
$$
Since $\chi$ is decreasing on $1+k^{\frac 1{1-k}}<r<2$ , then $\frac R r\chi(\frac rR)\leq 2+2(k-1)k^{\frac{k}{1-k}}$.
Furthermore, one can easily see that
\begin{align}
    |\Phi_{2, R}|\gtrsim 1.\label{2.8}
\end{align}

Combing \eqref{2.8} with \eqref{491}, we get
\begin{align}
|\Phi_{2, R}|\backsim 1.\label{2.9}
\end{align}

Hence, a straightforward computation gives
\begin{align}
\Big|\partial_r\Phi_{2, R}^{\alpha }\Big|&=\left|\partial_r\left(\frac{2}{N+4-b}\left[(4-b)\left(2-\chi'\left(\frac rR\right)\right)+(4N-4+b)\left(2-\frac R r\chi\left(\frac rR\right)\right)\right]\right)^{\alpha}\right|,\nonumber\\
& \lesssim |\Phi_{2, R}|^{\alpha-1}\left(\frac 1R\left|\chi''\left(\frac rR\right)\right|+\frac R{r^2}\left|\chi\left(\frac rR\right)\right|+\frac 1r\left|\chi'\left(\frac rR\right)\right|\right),\nonumber\\
& \lesssim\frac 1R.\label{2.13}
\end{align}

In a similar manner, we infer
\begin{align}
\Big|\partial_r^2\Phi_{2, R}^{\alpha}\Big|
&=\left|\partial_r^2\left(\frac{2}{N+4-b}\left[(4-b)\left(2-\chi'\left(\frac rR\right)\right)+(4N-4+b)\left(2-\frac R r\chi\left(\frac rR\right)\right)\right]\right)^{\alpha}
\right|,\nonumber\\
& \lesssim |\Phi_{2, R}|^{\alpha-2}|\partial_r\Phi_{2, R}^{\alpha}|^2,\nonumber \\&+\nonumber |\Phi_{2, R}|^{\alpha-1}\left(\frac 1{R^2}\left|\chi'''\left(\frac rR\right)\right|+\frac R{r^3}\left|\chi\left(\frac rR\right)\right|+\frac 1{r^2}\left|\chi'\left(\frac rR\right)\right|+\frac 1{rR}\left|\chi''\left(\frac rR\right)\right|\right),\nonumber\\
& \lesssim\frac 1{R^2}.\label{2.14}
\end{align}

Plugging estimates \eqref{2.13} and \eqref{2.14} together, we get
\begin{align*}
|\Delta \Phi_{2, R}^{\alpha}|
\lesssim\frac 1{R^2}.
\end{align*}
This completes the proof of this lemma.
\end{proof}

\section{Proof of Theorem \ref{BLOW1}}
\label{S3}

Define the Morawetz potential
\begin{align}\label{381}
\mathcal{Z}_R(t):=\mathcal{Z}_{\phi_R}(t)=2{\Im}\left(\int_{\R^N} \nabla\phi_R\cdot\nabla u(t)\bar u(t)dx\right).
\end{align}

\begin{prop}
\label{ZR-diffIneq}
We have 
\begin{align}
\mathcal{Z}_R'(t)&\leq 16E[u_0]-8\nu\|\nabla u(t)\|^2+C(R^{-2}+ R^{-b})\|\nabla u(t)\|^2{+O(R^{-2})}.\label{zr'2}
\end{align}
\end{prop}

\begin{proof}
 Using the estimate \eqref{nabla-phi-R} and the mass conservation, we obtain
\begin{align}
\displaystyle \left|\int_{\R^N}\Delta^2\phi_R|\nabla u|^2\,dx\right|+\left|\int_{\R^N}\partial_{jk}\Delta \phi_R\partial_ju\partial_k\bar u\,dx\right|&\lesssim R^{-2}\|\nabla u(t)\|_{L^2(|x|\geq R)}^2,\label{mr1}\\
\left|\int_{\R^N}(\Delta^3\phi_R)|u|^2\,dx\right|&\lesssim R^{-4}.\label{mr2}
\end{align}
Moreover, using \eqref{rd2} and $\nabla(|x|^{-b})=-b|x|^{-b-2}x$, we infer
 $$\nabla\phi_R\cdot\nabla(|x|^{-b})=-b|x|^{-b}\frac{\partial_r\phi_R}r.$$ Therefore
\begin{align*}
\mathcal{Z}_R'(t)
=&8\sum_{i=1}^N\int_{\R^N}\frac{\partial_r\phi_R}{r}|\nabla u_i|^2dx+8\sum_{i=1}^N\int_{\R^N}(\frac{\partial_r^2\phi_R}{r^2}-\frac{\partial_r\phi_R}{r^3})|x\cdot\nabla u_i|^2 dx\\
&+4\nu\int_{\R^N}\frac{\partial_r\phi_R}{r}|\nabla u|^2dx+4\nu\int_{\R^N}(\frac{\partial_r^2\phi_R}{r^2}-\frac{\partial_r\phi_R}{r^3})|x\cdot\nabla u|^2 dx\\
&+\frac{8-2b}{N+4-b}\int_{\R^N}\Big[-\partial_r^2\phi_R-(N-1+\frac{2Nb}{8-2b})\frac{\partial_r\phi_R}{r}\Big]|x|^{-b}|u|^{\frac{8-2b}{N}+2}dx\\
&-2\int_{\R^N}\Delta^2\phi_R|\nabla u|^2dx+O(R^{-2}).
\end{align*}

Owing to \eqref{mr1}, we infer that 
\begin{align}
\mathcal{Z}_R'(t)\leq 16E[u_0]-8\nu\|\nabla u(t)\|^2+C R^{-2}\|\nabla u(t)\|^2+\mathcal{R}_{1,\nu}+\mathcal{R}_2{+O(R^{-2})},\label{zr'}
\end{align}

where
\begin{align}
\mathcal{R}_{1,\nu}&=8\sum_{i=1}^N\int_{\R^N}(\frac{\partial_r\phi_R}{r}-2)|\nabla u_i|^2dx+8\sum_{i=1}^N\int_{\R^N}(\frac{\partial_r^2\phi_R}{r^2}-\frac{\partial_r\phi_R}{r^3})|x\cdot\nabla u_i|^2 dx\nonumber\\
&+4\nu\int_{\R^N}(\frac{\partial_r\phi_R}{r}-2)|\nabla u|^2dx+4\nu\int_{\R^N}(\frac{\partial_r^2\phi_R}{r^2}-\frac{\partial_r\phi_R}{r^3})|x\cdot\nabla u|^2 dx,\label{r1nu}
\end{align}
and
\begin{align}
\mathcal{R}_2
&=\frac{8-2b}{N+4-b}\int_{\R^N}\Big[\frac{16N}{8-2b}-\partial_r^2\phi_R-(N-1+\frac{2Nb}{8-2b})\frac{\partial_r\phi_R}{r}\Big]|x|^{-b}|u|^{\frac{8-2b}{N}+2}dx\nonumber\\
&=\int_{\R^N}\Phi_{2, R}(x)\,|x|^{-b}\,|u|^{\frac{8-2b}{N}+2}dx
.\label{r3}
\end{align}
 Clearly by \eqref{20233213}, we have $\mathcal{R}_{1,\nu}\leq 0$. Moreover, for later purpose, we also have
\begin{align}
\mathcal{R}_{1,\nu}&\leq-\sum_{i=1}^N\int_{|x|>R}\Phi_{1, R}|\nabla u_i|^2dx-\frac{\nu}{2}\int_{|x|>R}\Phi_{1, R}|\nabla u|^2dx,\label{r1}\\
\mathcal{R}_2&=\int_{|x|>R}\Phi_{2, R}(x)\,|x|^{-b}|u|^{\frac{8-2b}{N}+2}dx
\leq  C R^{-b}\|\nabla u\|^2.\label{r2}
\end{align}

Hence, \eqref{zr'} via \eqref{r1}, \eqref{r2} leads to \eqref{zr'2} as desired.
\end{proof}

Now, based on the above Proposition, we are in a position to give the proof of Theorem \ref{BLOW1}.\\

\noindent{(i)} {\bf The case $\nu>0$.}\\
Taking $R\gg 1$ in \eqref{zr'2}, we get
\begin{align}
\mathcal{Z}_R'(t)&\lesssim E[u_0]-\|\nabla u(t)\|^2.\label{zr'3}
\end{align}

On one hand, integrating in time \eqref{zr'3}, one can choose $T>0$ such that 
\begin{align}
\mathcal{Z}_R(t)&<0,\quad\forall\; t\geq T\label{zr0};\\
\mathcal{Z}_R(t)&=\mathcal{Z}_R(T)+\int_T^t\mathcal{Z}_R'(s)\,ds<-C\int_T^t\|\nabla u(s)\|^2\,ds,\quad\forall\; t>T.\label{zr1}
\end{align}
On the other hand,
\begin{align}
|\mathcal{Z}_R(t)|&\leq \|\nabla\phi_R\|_{L^\infty}\|\nabla u(t)\|\|u(t)\|\lesssim \|\nabla u(t)\|.\label{zr2}
\end{align}
Thus, \eqref{zr1} and \eqref{zr2} imply that
\begin{align}
\mathcal{Z}_R(t)&<-C\int_T^t|\mathcal{Z}_R(s)|^2\,ds,\quad\forall\; t>T.\label{zr3}
\end{align}
Owing to \eqref{zr3}, we see that the function $y(t):=\displaystyle\int_T^t|\mathcal{Z}_R(s)|^2\,ds$ satisfies the differential inequality $y^2\lesssim y'$. Obviously, the above nonlinear differential inequality has no global solutions.  This ends the proof when $\nu>0$.\\

\noindent{(ii)} {\bf The case $\nu=0$.}\\
We will focus on dimensions $N\geq 5$. The proof for dimensions $N\leq4$ follows the same lines by using \eqref{n=1}, \eqref{n=2}, \eqref{n=3} and \eqref{n=4} instead of \eqref{n=5}.
Owing to \eqref{n=5} and using Young's inequality, we infer
\begin{align}
\mathcal{R}_2&\lesssim R^{-b}\int_{|x|>R}\Phi_{2, R}|u|^{\frac{8-2b}{N}+2}dx\nonumber\\
&\lesssim R^{-b/2}\Big(\|\Delta(\Phi_{2, R}^{\frac{2}{4-b}})u\|+\|\nabla(\Phi_{2, R}^{\frac{2}{4-b}})\cdot\nabla u\|+\|\Phi_{2, R}^{\frac{2}{4-b}}\Delta u\|\Big)^{\frac{4-b}{2}} R^{-b/2}\|u\|^{\frac{8-2b}{N}+\frac b2}\nonumber\\
&\lesssim\frac{1}{R^{\frac{2b}{4-b}}}\Big(\|\Delta(\Phi_{2, R}^{\frac{2}{4-b}})u\|+\|\nabla(\Phi_{2, R}^{\frac{2}{4-b}})\|_{L^\infty}\|\nabla u\|_{L^2(|x|\geq R)}+\|\Phi_{2, R}^{\frac{2}{4-b}}\Delta u\|\Big)^2+\frac{1}{R^2}\|u_0\|^{\frac{4(8-2b)}{Nb}+2}.\label{3222}
\end{align}
We will treat the terms in RHS of \eqref{3222} one by one. By \eqref{clm} and the conservation of mass, we have
\begin{align}\label{127-1}
\|\Delta(\Phi_{2, R}^{\frac{2}{4-b}})u\|^2 \lesssim\frac 1{R^4}.
\end{align}
Moreover, as in the proof of  \eqref{-12.9} and \eqref{2.13}, we obtain
\begin{align}\label{3224}
|\nabla(\Phi_{2, R}^{\frac{2}{4-b}})|=|\partial_r(\Phi_{2, R}^{\frac{2}{4-b}})|\lesssim\frac 1{R}.
\end{align}

Hence, by \eqref{GNI-R}, \eqref{3224} and the conservation of mass, we infer 
\begin{align}\label{1272}
\|\nabla(\Phi_{2, R}^{\frac{2}{4-b}})\|_{L^\infty}^2\|\nabla u\|_{L^2(|x|\geq R)}^2&\lesssim
\frac{1}{R^{2}}{\|\nabla u\|_{L^2(|x|\geq R)}^2}\nonumber\\
&\lesssim \frac{1}{R^{2}}+\frac{1}{R^{2}}{\|\nabla u\|_{L^2(|x|\geq R)}^4}\nonumber\\
&\lesssim\frac{1}{R^{2}}+\frac{1}{R^{2}}{\|\Delta u\|_{L^2(|x|\geq R)}^2}.
\end{align}

Further, by \eqref{3222}, \eqref{127-1} and \eqref{1272} and the conservation of mass, we get
\begin{align}
\mathcal{R}_2\lesssim \frac{1}{R^{\frac{2b}{4-b}}}\|\Phi_{2, R}^{\frac{2}{4-b}}\Delta u\|^2+\frac{1}{R^{2+\frac{2b}{4-b}}}\|\nabla u\|_{L^2(|x|\geq R)}^2+\frac{1}{R^2}+\frac{1}{R^{4+\frac{2b}{4-b}}}.\label{2.18}
\end{align}
Furthermore, owing to \eqref{GNI-R} and using  Young's inequality, we obtain
\begin{align}
 CR^{-2}\|\nabla u\|_{L^2(|x|\geq R)}^2
\leq C R^{-2}+CR^{-2}\|\Delta u\|_{L^2(|x|\geq R)}^2.\label{2.20}
\end{align}
Noticing that $|\Delta u|^2\lesssim \sum_{i=1}^N|\nabla u_i|^2$ and owing to \eqref{fi1r}, \eqref{r1}, \eqref{2.18}, and \eqref{2.20}, we have
\begin{align}
&CR^{-2}\|\nabla u\|_{L^2(|x|\geq R)}^2+\mathcal{R}_{1,0}+\mathcal{R}_2\nonumber\\
\leq&\sum_{i=1}^N\int_{|x|>R}\Big(C\frac{1}{R^{\frac{2b}{4-b}}}\Phi_{2, R}^{\frac{4}{4-b}}+C\frac{1}{R^{2+\frac{2b}{4-b}}}+C\frac1{R^6}-\Phi_{1, R}\Big)|\nabla u_i|^2dx\nonumber\\
&+\frac{C}{R^2}+\frac{C}{R^{2+\frac{2b}{4-b}}}+\frac{C}{R^{4+\frac{2b}{4-b}}}.\label{3224'}
\end{align}
{Next, we claim that for any $r>R\gg 1$,
\begin{align}\label{3225}
C\frac{1}{R^{\frac{2b}{4-b}}}\Phi_{2, R}^{\frac{4}{4-b}}+C\frac{1}{R^{2+\frac{2b}{4-b}}}+C\frac1{R^6}-\Phi_{1, R}\leq 0.
\end{align}}
$\bullet$ \underline{\bf First case}: $R<r\leq R(1+k^{\frac{1}{1-k}})$. 
A straightforward computation gives
\begin{align*}
C&\frac{1}{R^{\frac{2b}{4-b}}}\Phi_{2, R}^{\frac{4}{4-b}}(r)-\Phi_{1, R}(r)\nonumber\\
&=C\frac{1}{R^{\frac{2b}{4-b}}}(\frac rR-1)^{\frac{4(k-1)}{4-b}}\Big(\frac{8}{N+4-b}[k(4-b)+(4N-2+b)(1-\frac Rr)]\Big)^{\frac {4}{4-b}}-16\frac Rr(\frac rR-1)^k\nonumber\\
&=(\frac rR-1)^k\Big[C\frac{1}{R^{\frac{2b}{4-b}}}(\frac rR-1)^{\frac{4(k-1)}{4-b}-k}\Big(\frac{8}{N+4-b}[k(4-b)+(4N-2+b)(1-\frac Rr)]\Big)^{\frac {4}{4-b}}-16\frac Rr\Big].
\end{align*}
Since $\frac Rr\geq \frac 12 $, we get
\begin{align}
C&\frac{1}{R^{\frac{2b}{4-b}}}\Phi_{2, R}^{\frac{4}{4-b}}(r)-\Phi_{1, R}(r)\nonumber\\
&\leq(\frac rR-1)^k\Big[C\frac{k^{\frac{k}{k-1}-\frac{4}{4-b}}}{R^{\frac{2b}{4-b}}}\Big(\frac{8}{N+4-b}[k(4-b)+(2N-1+\frac b2)]\Big)^{\frac {4}{4-b}}-\frac {16}{1+k^{\frac{1}{1-k}}}\Big]\nonumber\\
&\leq k^{\frac{k}{1-k}}\Big[C\frac{k^{\frac{k}{k-1}-\frac{4}{4-b}}}{R^{\frac{2b}{4-b}}}\Big(\frac{8}{N+4-b}[k(4-b)+(2N-1+\frac b2)]\Big)^{\frac {4}{4-b}}-\frac {16}{1+k^{\frac{1}{1-k}}}\Big]\label{3226''}\nonumber\\
&\leq C\frac{k^{-\frac{4}{4-b}}}{R^{\frac{2b}{4-b}}}\Big(\frac{8}{N+4-b}[k(4-b)+(2N-1+\frac b2)]\Big)^{\frac {4}{4-b}}-\frac{16k^{\frac{1}{1-k}}}{1+k^{\frac{1}{1-k}}}\nonumber\\
&\leq C\frac{1}{R^{\frac{2b}{4-b}}}\Big(\frac{8}{N+4-b}[(4-b)+(2N-1+\frac b2)]\Big)^{\frac {4}{4-b}}-\frac{16k^{\frac{1}{1-k}}}{1+k^{\frac{1}{1-k}}}.
\end{align}

By taking $R$ large enough, we can obtain
\begin{align}
C\frac{1}{R^{\frac{2b}{4-b}}}\Phi_{2, R}^{\frac{4}{4-b}}(r)-\Phi_{1, R}(r)
\leq -\frac{8k^{\frac{1}{1-k}}}{1+k^{\frac{1}{1-k}}},\label{3226}
\end{align}
and 
\begin{align}\label{942}
C\frac{1}{R^{2+\frac{2b}{4-b}}}+C\frac1{R^6}-\frac{8k^{\frac{1}{1-k}}}{1+k^{\frac{1}{1-k}}}\leq 0.
\end{align}

Hence, by the above estimates \eqref{3226} and \eqref{942}, we get the claim \eqref{3225} in this case.

$\bullet$ \underline{\bf Second case}: $R(1+k^{\frac{1}{1-k}})<r{\leq 2R}$. Since $\chi'<0$, then
\begin{align*}
\frac R r\chi\left(\frac rR\right)\leq \frac{1}{1+k^{\frac{1}{1-k}}}\chi(1+k^{\frac{1}{1-k}})=2-2\frac{k^{\frac{k}{1-k}}}{1+k^{\frac{1}{1-k}}}.
\end{align*}

Hence,
\begin{align}\label{9481}
\Phi_{1, R}(r)=8\left(2-\frac R r\chi\left(\frac rR\right)\right)\geq 16\frac{k^{\frac{k}{1-k}}}{1+k^{\frac{1}{1-k}}}.
\end{align}

Moreover, by \eqref{2.9}, we have $|\Phi_{2, R}|\lesssim 1.$ Therefore, taking $R$ large enough, we infer
\begin{align}\label{3227}
C\frac{1}{R^{\frac{2b}{4-b}}}\Phi_{2, R}^{\frac{4}{4-b}}(r)+C\frac{1}{R^{2+\frac{2b}{4-b}}}+C\frac1{R^6}-\Phi_{1, R}(r)\leq 0.
\end{align}

$\bullet$ \underline{\bf Third case}: $r\geq 2R$. Noting that $\phi_R(r)=0$, we have $\Phi_{1, R}(r)=16$ and $\Phi_{2, R}(r)=\frac{16N}{N+4-b}$. By taking $R$ large enough, we have
\begin{align}\label{461}
C\frac{1}{R^{\frac{2b}{4-b}}}\Phi_{2, R}^{\frac{4}{4-b}}(r)+C\frac{1}{R^{2+\frac{2b}{4-b}}}+C\frac1{R^6}-\Phi_{1, R}(r)\leq 0.
\end{align}
This proves \eqref{3225}.

Now, taking advantage of \eqref{3224'} and \eqref{3225} and choosing $R$ large enough, we get
\begin{equation}
\mathcal{Z}_R'(t)
\leq 16E[u_0]+\frac{C}{R^2}+\frac{C}{R^{2+\frac{2b}{4-b}}}+\frac{C}{R^{4+\frac{2b}{4-b}}}\leq 8E[u_0] <0.\label{3.35'}
\end{equation}

With inequality \eqref{3.35'} established, we can now prove the second part of Theorem \ref{BLOW1}. Assume that $T^* = \infty$. Then, by \eqref{3.35'}, there exists a time $t_1 > 0$ such that$\mathcal{Z}_R(t)<0$ for any $t\geq t_1$. Hence, by the Cauchy-Schwarz inequality via \eqref{381}, and integrating \eqref{3.35'} in $[t_1, t]$,  we get for all $t\geq t_1$,
\begin{align}
-\sqrt{\|\Delta u(t)\|}\lesssim \mathcal{Z}_R(t)\lesssim E[u_0](t-t_1).\label{3.36}
\end{align}

Finally, \eqref{3.36} gives for large time
\begin{equation*}
{\|\Delta u(t)\|}\gtrsim t^2.
\end{equation*}
This completes the proof of the second part of Theorem \ref{BLOW1}.

\vspace{0.3cm}

\hrule

\vspace{0.3cm}

{\noindent {\bf\large Acknowledgements.} {\it The research of R. Bai was partially supported by the NSF of China (Grant No. 12401302), and the Postdoctoral Fellowship Program of CPSF (Grant No. GZC20230694).}

\vspace{0.3cm}

{\noindent{\bf\large Declarations.}}
On behalf of all authors, the corresponding author states that there is no conflict of interest. No data-sets were generated or analyzed during the current study.

\vspace{0.3cm}

\hrule



\end{document}